\documentclass[12pt]{amsart}
\usepackage{hyperref}

\usepackage{amsthm}
\usepackage{amssymb,mathrsfs,mathcomp,amsfonts,tabularx, hhline, textcomp}

\usepackage{upgreek}
\usepackage[matrix,arrow,curve]{xy}

\usepackage{graphicx}
\newcommand{\comp}{\mathbin{\scriptstyle{\circ}}}
\newcommand{\g}{{\mathrm{g}}}
\newcommand{\B}{{\mathbf B}}
\newcommand{\CC}{\mathbb{C}}

\newcommand{\QQ}{\mathbb{Q}}

\newcommand{\ZZ}{\mathbb{Z}}
\newcommand{\PP}{\mathbb{P}}

\newcommand{\MMM}{{\mathscr{M}}}

\newcommand{\DP}{\operatorname{DP}_{5}^{\mathrm A_4}}
\newcommand{\qq}{\mathbin{\sim_{\scriptscriptstyle{\QQ}}}}
\newcommand{\qW}{\operatorname{q}_{\mathrm{W}}}
\newcommand{\qQ}{\operatorname{q_{\QQ}}}

\newcommand{\lcm}{\operatorname{lcm}}
\newcommand{\Pic}{\operatorname{Pic}}

\newcommand{\Cl}{\operatorname{Cl}}

\newcommand{\ind}{\mathrm{{r}}}
\newcommand{\tors}{\mathrm{t}}
\newcommand{\mumu}{{\boldsymbol{\mu}}}
\makeatletter
\providecommand*{\approxident}{\mathrel{\mathpalette\@approxident\sim}} 
\newcommand*{\@approxident}[2]{\sbox0{$#1\vcenter{}$}\sbox2{$\m@th#1\equiv$}
\dimen2=\dimexpr\ht2 - \ht0\relax\sbox4{$\m@th#1\sim$}
\dimen4=\dimexpr\ht4 - \ht0\relax\dimen0=\dimexpr-\ht4 - \dp4 + \dimen2 
\relax\vcenter{\offinterlineskip\copy4 \kern\dimen0 \copy4 \kern\dimen0 \copy4 \ifdim\dp4=\z@\kern\dimexpr -\ht0 + \dimen4\relax\fi}} 
\makeatother

\newcommand{\xref}[1]{{\rm~\ref{#1}}}

\newcounter{NN}\numberwithin{NN}{section}
\renewcommand{\theNN}{\arabic{NN}${}^o$}
\def\nr{\refstepcounter{NN}{\theNN}}

\makeatletter
\@addtoreset{equation}{subsection}
\renewcommand{\thefigure}{\arabic{section}.\arabic{subsection}.\@arabic\c@figure}
\@addtoreset{figure}{subsection}
\makeatother

\swapnumbers
\theoremstyle{plain}
\newtheorem{theorem}[subsection]{Theorem}
\newtheorem{lemma}[subsection]{Lemma}

\newtheorem{proposition}[subsection]{Proposition}

\newtheorem{corollary}[subsection]{Corollary}
\newtheorem{scorollary}[equation]{Corollary}
\newtheorem*{claim*}{Claim}

\newtheorem{slemma}[equation]{Lemma}

\theoremstyle{definition}

\newtheorem*{definition*}{Definition}

\newtheorem{example-remark}[subsection]{Remark-Example}
\newtheorem{subexample-remark}[equation]{Remark-Example}

\newtheorem{scase}[equation]{}

\newtheorem*{notation*}{Notation}

\newtheorem{examples}[subsection]{Examples}

\newtheorem{sremark}[equation]{Remark}

\author{Yuri Prokhorov}

\title{Rationality of $\mathbb{Q}$-Fano threefolds of large Fano index}
\dedicatory{To Miles Reid on his 70th birthday}
\address{Yuri~Prokhorov:
Steklov Mathematical Institute of Russian Academy of Sciences, Moscow, Russia
\newline\indent
Department of Algebra, 
Moscow Lomonosov University, Russia 
\newline\indent
National Research University Higher School of Economics, Russia 
}
\email{prokhoro@mi-ras.ru}

\thanks{
The author was partially supported by the HSE University Basic Research Program, Russian Academic Excellence Project '5-100'.
}
\begin{document}
\begin{abstract}
We prove that $\QQ$-Fano threefolds of Fano index $\ge 8$ are rational. 
\end{abstract}
\maketitle
\section{Introduction}
Recall that a projective algebraic variety $X$
called \textit {$\QQ$-Fano} if it has only
terminal $\QQ$-factorial singularities,
$\Pic (X)\simeq\ZZ$,
and the anticanonical divisor $-K_X$ is ample. 
$\QQ$-Fano varieties plays a very important role in the higher dimensional geometry
since they appears naturally in the minimal model program as building blocks in so-called Mori fiber spaces. 
It is known that $\QQ$-Fano varieties of given dimension are bounded, i.e. they form an algebraic family \cite{Kawamata-1992bF}, \cite{Birkar2016}.
Moreover, the method of \cite{Kawamata-1992bF} allows to produce a finite but very huge list of numerical candidates (Hilbert series) of $\QQ$-Fanos
\cite{GRD}.
In dimension three there are a lot of classificational results of $\QQ$-Fanos of special types (see e.g. \cite{Sano-1996}, \cite{Suzuki-2004}, \cite{Takagi-2006}, \cite{Prokhorov2008a}, \cite{Brown-Kerber-Reid-codim4}, \cite{Prokhorov-Reid})
but the full classification is very far from being complete. 

An important invariant of a $\QQ$-Fano variety $X$ is its 
\emph{$\QQ$-Fano index} $\qQ(X)$ which is the maximal integer $q$ such that $-K_X\qq qA$ for some integral Weil divisor $A$, where $\qq$ defines the $\QQ$-linear equivalence. 
In this paper we prove the following.

\begin{theorem}\label{theorem:main}
Let $X$ be a $\QQ$-Fano threefold with $\qQ(X)\ge 8$. Then $X$ is rational. 
\end{theorem}
Note that in some sense our result is optimal: according to \cite{Okada2019} a very general 
weighted hypersurface $X_{14}\subset \PP(2, 3, 4, 5, 7)$ is a non-rational (and even non-stably rational) $\QQ$-Fano threefold with $\qQ(X)=7$. On the other hand, 
the result of Theorem~\ref{theorem:main} can be essentially improved.
We hope that non-rational $\QQ$-Fano threefolds of large indices admit a reasonable classification.

The structure of the paper is as follows. Section~\ref{sect:prel} is preliminary. In Section~\ref{sect:torsion} we list certain kinds of $\QQ$-Fano threefolds with torsions in the Weil divisor class group $\Cl(X)$. In Section~\ref{sect:constr} the main birational construction is introduced. The proof of the main theorem is given in Sections~\ref{sect:7}-\ref{sect:8} by 
case by case analysis.

\section{Preliminaries}
\label{sect:prel}
We work over the complex number field $\CC$ throughout.
\subsection{Notation}

\begin{itemize}
\item 
$\Cl(X)$ denotes the Weil divisors class group of a normal variety;
\item 
$\Cl(X)_\tors$ denotes the torsion part of $\Cl(X)$;
\item 
$\B(X)$ is the basket of a terminal threefold $X$ \cite{Reid-YPG1987}; 
\item 
$\ind(X,P)$ is the singularity index of a terminal point $P\in X$;
\item
$\g(X):=\dim |-K_X|-1$ is the genus of a $\QQ$-Fano threefold $X$.
\end{itemize}
For a $\QQ$-Fano threefold $X$ we define its \emph{Fano} and
\emph{$\QQ$-Fano index} by:
\begin{align*}
\qW(X)&:= \max \{ q\in\ZZ \mid \hbox{$-K_X\sim qA$ with $A$ a Weil divisor}\},
\\[3pt]
\qQ(X)&:= \max \{ q\in \ZZ \mid \hbox{$-K_X\qq qA$ with $A$ a Weil divisor}\},
\end{align*}
where $\sim$ (resp. $\qq$) is the linear (resp. $\QQ$-linear) equivalence.
Clearly, $\qW(X)$ divides $\qQ(X)$, and $\qW(X)=\qQ(X)$ unless
$K_X+qA\in\Cl X$ is a nontrivial torsion element. 
Throughout this paper, for a $\QQ$-Fano threefold $X$, by $A$ we denote a Weil divisor such that $-K_X\qq \qQ(X) A$. If $\qQ(X)=\qW(X)$ we take $A$ so that $-K_X\sim \qW(X) A$.

\begin{theorem}[{\cite{Suzuki-2004}}]
Let $X$ be a $\QQ$-Fano threefold. Then 
\begin{equation}
\label{equation-index}
\qQ(X)\in \{1,\dots, 11, 13,17,19\}
\end{equation}
and all the possibilities do occur.
\end{theorem}

The following easy observation will be used freely.

\begin{lemma}[{\cite[Lemma~5.1]{Kawamata-1988-crep}}] 
\label{lemma:K-index}
Let $(X\ni P)$ be a threefold terminal singularity and let $\Cl^{\mathrm{sc}}(X,P)$ be the subgroup of the (analytic) Weil divisor class group consisting of .
Weil divisor classes which are $\QQ$-Cartier. Then the group $\Cl^{\mathrm{sc}}(X,P)$ is cyclic of order $\ind(X,P)$
and is generated by the canonical class $K_X$. 
\end{lemma}

\begin{lemma}\label{lemma:qQ=qW}
Let $X$ be a $\QQ$-Fano threefold and let $\ind(X)$ be the global Gorenstein index of $X$. 
Then the equality $\qQ(X)=\qW(X)$ holds if and only if $\qQ(X)$ and $\ind(X)$ are coprime.
\end{lemma}

\begin{proof}
The ``only if'' part of the statement immediately follows from Lemma~\ref{lemma:K-index} (see \cite[Lemma 1.2(3)]{Suzuki-2004}).
Let us prove the ``if'' part. So, we assume that $\gcd(\qQ(X), \ind(X))=1$.
Put $q:=\qQ(X)$ and write $-K_X\qq q A'$, where $A'$ is a Weil divisor. Then $\Xi:=K_X+qA'$ is a torsion element in $\Cl(X)$. 
Take $A=A' +t\Xi$, $t\in \ZZ$. Then 
\begin{equation*}
K_X+qA\sim (1+qt)\Xi. 
\end{equation*}
Since the order of $\Xi$ in $\Cl(X)$ divides $\ind(X)$, there exists $t\in \ZZ$ such that $(1+qt)\Xi\sim 0$. 
\end{proof}

The following proposition a consequence of the classification of $\QQ$-Fano threefolds of large degree (see \cite{Prokhorov-2007-Qe}, \cite{Prokhorov2008a}, \cite{Prokhorov-2013-fano}).

\begin{proposition}
\label{prop:rat}
Let $X$ be a $\QQ$-Fano threefold with $\qQ(X)=\qW(X)\ge 3$. 
Assume that $X$ is not rational. Then $X$ belongs to one of the following classes below.
\par\smallskip\noindent
\begin{center}
{\rm 
\renewcommand{\tabcolsep}{11pt} 
\begin{tabularx}{\textwidth}{|X|X|r|r|r|r|r|}
\hline
&&\multicolumn{5}{|c|}{$\dim |kA|$}
\\\hhline{|~|~|-----}
$\qQ$&$\g(X)$& $|A|$ & $|2A|$ & $|3A|$ & $|4A|$ & $|5A|$ 
\\\hline
$13$ & $4$ & $-1$ & $-1$ & $0$& $0$& $0$ 
\\
$11$ & $\le 9$ & $\le 0$ & $0$ & $0$ & $1$ & $\le 2$
\\ 
$9$ & $4$ & $-1$& 0& 0& 1& 1
\\
$8$ & $\le 10$ & $\le 0$& 0& $\le 1$& $\le 2$& $\le 3$ 
\\
$7$ & $\le 14$ & $\le 0$ & $\le 1$ & $\le 2$ & $\le 4$ & $\le 6$
\\
$6$ & $\le 15$ & $\le 0$ & $\le 1$ & $\le 3$ & $\le 6$ & $\le 11$ 
\\
$5$ & $\le 18$ & $\le 1$ & $\le 3$ & $\le 7$ & $\le 12$ &
\\
$4$ & $\le 21$ & $\le 1$ & $\le 5 $ & $\le12$ &&
\\
$3$ & $\le 20$ & $\le 2$ & $\le 9$ &&& \\\hline
\end{tabularx}}
\end{center}
\end{proposition}

\begin{proof}
Given $q$, $\QQ$-Fano threefolds $X$ with $\qQ(X)=q$ and genus $\g(X)\ge \g_q$ are completely described in \cite{Prokhorov2008a}, \cite{Prokhorov-2013-fano}, \cite{Prokhorov-e-QFano7}, where the number $\g_q$ is given by the third column in the table. It is easy to see that all these varieties are rational. The rest can be checked by a computer search 
as explained in \cite{Suzuki-2004}, \cite[Lemma~3.5]{Prokhorov2008a} or \cite[2.4]{Prokhorov-Reid} (see also \cite{GRD}).
\end{proof}

\begin{proposition}[{\cite{Kawamata-1996}}, {\cite{Kawakita2005}}]
\label{prop:discrepancies}
Let $Y\ni P$ be a threefold terminal point of index $r>1$ and 
let 
\begin{equation*}
f: (\tilde Y\supset E)\to (Y\ni P)
\end{equation*}
be a divisorial Mori extraction,
where $E$ is the exceptional divisor and $f(E)=P$. 
Write 
\begin{equation*}
K_{\tilde Y}=f^*K_Y+\alpha E.
\end{equation*}
Then the following assertions hold.
\begin{enumerate}
\item \label{theorem-Kawamata-blowup}
If $Y\ni P$ is cyclic quotient singularity of type
$\frac1r (1, a, r-a)$, then $\alpha=1/r$ and $f$ is a weighted blowup with weights $(1, a, r-a)$.
\item 
If $Y\ni P$ is a point of type other than $\mathrm{cA}/r$ and $r>2$, then 
$\alpha=1/r$.
\item 
If $Y\ni P$ is of type $\mathrm{cA}/r$
and its basket $\B(Y,P)$ consists of $m$ points of index $r$, then $\alpha=a/r$, where $m\equiv 0\mod a$.
\end{enumerate}
\end{proposition}

\section{$\QQ$-Fano threefolds with torsion in the divisor class group}
\label{sect:torsion}
\subsection{}\label{notation:torsion}
Let $X$ be a $\QQ$-Fano threefold and let $\Xi\in \Cl(X)_\tors$ be a non-trivial torsion element of order $n$.
Then $\Xi$ defines a finite \'etale in codimension two cover 
$\pi: X'\to X$ such that $X'$ has only terminal singularities, $K_{X'}=\pi^* K_X$ and $\pi^*\Xi=0$
(see \cite[3.6]{Reid-YPG1987}). Clearly, $X'$ is a Fano variety. However, in general, we cannot say that $X'$ is $\QQ$-factorial neither $\Pic(X')\simeq \ZZ$. Let $q:=\qQ(X)$. Take $A$ so that $-K_X\qq qA$ and let $A':=\pi^*A$.
Then $-K_{X'}\qq q A'$. Hence, $\qQ(X')$ is divisible by $q$.

\begin{sremark}\label{remark:torsion}
In the above notation, assume that $q\ge 5$. Run the MMP on $X'$. On each step the relation $-K_{X'}\qq q A'$ is preserved. Therefore, at the end we obtain a $\QQ$-Fano threefold $X''$ such that $-K_{X''}\qq q A''$, where $q\ge 5$. Then by~\eqref {equation-index} we have $\qQ(X'')=q$ and so $\qQ(X')=q$. Moreover, 
\[
\g(X'')\ge \g(X').
\]
\end{sremark} 

\begin{proposition}\label{proposition:qQqW}
Notation as in~\xref{notation:torsion}. 
Assume that $q\ge 3$ and $q\neq \qW(X)$. Take $\Xi:=K_X+qA$.
Then 
\begin{equation}\label{eq:qQqW}
(q,n)=(3,3)\quad \text{or}\quad (4,2).
\end{equation}
\end{proposition}

\begin{proof}
As in Proposition~\ref{prop:rat} we use a computer search.
But in this case the algorithm should be modified as follows (cf. \cite{Caravantes2008}).
For short, we denote $r_P:=\ind(X,P)$. 
Let $r:=\lcm(\{r_P\})$ be the global Gorenstein index of~$X$.

\subsection*{Step 1} 
By \cite{Kawamata-1992bF} we have the inequality
\begin{equation*}
0<-K_X\cdot c_2(X)= 24-\sum_{P\in \B} \frac{r_P-1}{r_P}.
\end{equation*}
This produces a
finite (but huge) number of possibilities for the basket $\B(X)$ and the number
$-K_X\cdot c_2(X)$. 

\subsection*{Step 2.}
\eqref{equation-index} implies that $q\in \{3,\dots,11, 13,17,19\}$.
In each case we compute $A^3$ by the formula
\begin{equation*}
A^3=\frac{12}{(q-1)(q-2)}\Bigl(
1-\frac{A\cdot c_2(X)}{12}+\sum_{P\in B} c_P(-A)
\Bigr)
\end{equation*}
(see \cite{Suzuki-2004}), where $c_P$ is the correction term in the
orbifold Riemann-Roch formula \cite{Reid-YPG1987}. The number $rA^3$
must be a positive integer \cite[Lemma~1.2]{Suzuki-2004}.

\subsection*{Step 3.}
Next, by \cite[Prop.\ 2.2]{Suzuki-2004} the Bogomolov--Miyaoka inequality (see \cite{Kawamata-1992bF})
implies that
\begin{equation*}
\left(4q^2 - 3q\right)A^3\le -4K_X\cdot c_2(X).
\end{equation*}

\subsection*{Step 4.} 
In a neighborhood of each point $P\in X$ we can write 
$A\sim l_PK_X$ by Lemma~\ref{lemma:K-index}, where $0\le l_P<r_P$.
There is a finite number of possibilities for the collection $\{(l_P)\}$.

\subsection*{Step 5.}
The number $n$ is determined as minimal positive such that $\chi(n\Xi)=1$ (by the Kawamata--Viehweg vanishing).
Hence, $n$ can be computed by using
orbifold Riemann-Roch.

\subsection*{Step 6.}
Finally, applying Kawamata--Viehweg vanishing we obtain
\begin{equation*}
\chi(tA+s\Xi)=h^0(tA+s\Xi)=0.
\end{equation*}
for $-q<t<0$ and $0\le s<n$. Again, we check this condition using
orbifold Riemann-Roch.

To run this algorithm the author used the computer algebra system PARI/GP \cite{PARI2}.
As the result, 
we get a short list from which one can see that \eqref{eq:qQqW} holds.
\end{proof}

\begin{proposition}\label{proposition:torsions}
Notation as in~\xref{notation:torsion}. 
Assume that $q\ge 5$ and $\Cl(X)_\tors$ contains an element $\Xi$ of order $n\ge 2$. 
Then $n\le 3$, $\qQ(X)=q$, and one of the following holds:
\par\smallskip\noindent
\begin{center}
{\rm 
\renewcommand{\tabcolsep}{4pt}
\begin{tabularx}{\textwidth}{|c|c|c|c|X|c|c|c|c|}
\hline
& $n$& $q$&$\g(X)$ &$\B(X)$& $A^3$& $\mathbf{k}$ & $\B(X')$&$\g(X')$
\\\hline
\nr\label{tor:1} & $2$& $5$& 2&$(2, 4, 14)$& $1/28$& $(1, 0, 7)$ &
$(4^2, 7)$ &4

\\
\nr\label{tor:2} & $3$& $5$& 3&$(2, 9, 9)$& $1/18$& $(0, 3, 6)$& 
$(2^3, 3^2)$ & 10

\\
\nr\label{tor:3} & $2$& $5$& 5&$(4, 4, 12)$& $1/12$& $(0, 2, 6)$ &
$(2, 4^2, 6)$ & 10

\\
\nr\label{tor:4} & $2$& $5$& 7&$(2, 2, 3, 14)$& $5/42$& $(0, 1, 0, 7)$ &
$(2^2, 3^2, 7)$ &14

\\
\nr\label{tor:6} & $2$& $5$& 10&$(2, 3, 4, 12)$& $1/6$& $(0, 0, 2, 6)$& 
$(2^3, 3^2, 6)$ & 20

\\
\nr\label{tor:5} & $2$& $5$& 8&$(2, 2, 4, 8)$& $1/8$& $(1, 1, 2, 4)$& 
$(2, 4)$ &16

\\
\nr\label{tor:7} & $2$& $5$& 11&$(2, 4, 4, 6)$& $1/6$& $(1, 2, 2, 3)$ &
$(2^2, 3)$ & 21

\\
\nr\label{tor:8} & $2$& $7$& 6&$(2, 6, 10)$& $1/30$& $(0, 3, 5)$ &
$(2^2, 3, 5)$ & 11

\\
\nr\label{tor:9} & $2$& $7$& 7&$(2, 2, 3, 4, 8)$& $1/24$& $(1, 1, 0, 2, 4)$ &
$(2, 3^2, 4)$ &14\\\hline
\end{tabularx}
}
\end{center}
\par\smallskip\noindent 
Moreover, the group $\Cl(X)_\tors$ is cyclic and generated by $\Xi$.
\end{proposition}

\begin{proof}
Similar to Proposition~\ref{proposition:qQqW}. But in this case, $\qQ(X)=\qW(X)$ and we have to modify one step:

\smallskip
\textbf{Step 4${}'$.} 
In this case $\gcd(q,r)=1$ by Lemma~\ref{lemma:qQ=qW}. Since $K_X+qA\sim 0$, the numbers $l_P$ are uniquely determined by $1+ql_P\equiv 0\mod r_P$.
But for $\Xi$ there are several choices. 
Again, near each point $P\in X$ we can write 
$\Xi\sim k_PK_X$ by Lemma~\ref{lemma:K-index}, where for the collection $\mathbf{k}=(\{k_P\})$ there are only a finite number of possibilities.

We obtain a list $\{(n, q, \B(X), \g(X), A^3, \mathbf{k})\}$. In each case we compute the basket $\B(X')$ of a (terminal) Fano threefold $X'$ with $A'^3=nA^3$. By Remark~\ref{remark:torsion} we have $\qQ(X')=q$. 
Then we can compute $\g(X')$ by orbifold Riemann-Roch.
At the end we get the list in the table and several extra possibilities which do not occur because 
$\g(X') \le 32$ in the case $\qQ(X')=5$ by \cite[Th. 1.2(v)]{Prokhorov-2013-fano} and Remark~\ref{remark:torsion}.
\end{proof}

We do not assert that all the possibilities in Proposition~\ref{proposition:torsions} occur.
We are able only to provide several examples for \ref{tor:2}, \ref{tor:5}-\ref{tor:9}. 

\begin{examples}
The following quotient of weighted hypersurfaces are $\QQ$-Fano threefolds as in \ref{tor:2}, \ref{tor:5}-\ref{tor:9}.

\begin{itemize}
\item[\ref{tor:2}]
$\{x_1^6+x_2^3+x_2'^3 +x_3x_3'=0\}\subset \PP(1, 2^2, 3^2)/\mumu_3 (0,1,-1,1,-1)$;
\item[\ref{tor:5}]
$\{x_1^6+x_1'^6+x_2x_4+x_3^2=0\}\subset \PP(1^2,2,3,4)/\mumu_2(0,1,1,1,1)$;
\item[\ref{tor:7}]
$\{x_1^4+x_1'^4+x_1x_3+x_2^2+x_2'^2=0\}\subset \PP(1^2,2^2,3)/\mumu_2(0,1,1,1,0)$;
\item[\ref{tor:8}] 
$\{x_1^8+x_2^4+x_3x_5+x_4^2=0\}\subset \PP(1,2,3,4,5)/\mumu_2(0,1,1,1,1)$;
\item[\ref{tor:9}]
$\{x_1^6+x_2x_4+x_3^2+x_3'^2=0\}\subset \PP(1,2,3^2,4)/\mumu_2(0,1,0,1,1)$. 
\end{itemize}
One can expect also that the variety \ref{tor:1} is a quotient of a codimension four $\QQ$-Fano (see \cite[No.~41418]{GRD} and \cite[\S~5.4]{Coughlan-Duca2018}).
\end{examples}

Using the orbifold Riemann-Roch one can compute dimensions of linear systems on $X$:

\begin{corollary}\label{cor:7tor}
In the cases \xref{tor:8} and \xref{tor:9} of Proposition~\xref{proposition:torsions} the dimension of the linear systems $|kA+s\Xi|$ are as follows
\begin{center}
{ 
\renewcommand{\tabcolsep}{6pt} 
\begin{tabularx}{\textwidth}{|X|rrrrrrr||rrrrrrr|}\hline
& \multicolumn{7}{c||}{$\text{\rm \ref{tor:8}}$}& \multicolumn{7}{c|}{$\text{\rm \ref{tor:9}}$}
\\\hline
$k$ & $1$& $2$& $3$& $4$& $5$& $6$& $7$& $1$& $2$& $3$& $4$& $5$& $6$& $7$
\\\hline
$\dim |kA|$ &$0$ & $0$&$ 0$&$ 1$&$ 2$&$ 4$&$ 6 $&$-1$&$ 0 $&$1$&$ 2$&$ 3$&$ 5$&$ 7$
\\\hline
$\dim |kA+\Xi|$&$ -1$&$ 0$&$ 1$&$ 2$&$ 3$&$ 4$&$ 5 $&$ 0$&$ 0$&$ 1$&$ 2$&$ 3$&$ 5$&$ 7$\\\hline
\end{tabularx}}
\end{center}
\end{corollary}

Combining \ref{proposition:torsions} and \ref{prop:discrepancies} we obtain.
\begin{corollary}\label{corollary:tor:discr}
Let $Y$ be a $\QQ$-Fano threefold with $\qQ(X)\ge 5$. Assume that $\Cl(Y)_\tors\neq 0$. Let $P\in Y$ be a non-Gorenstein point
and let $f$ be a divisorial Mori extraction of $P$. Then for the discrepancy $\alpha$ of the exceptional divisor $E\subset \tilde Y$ we have 
\begin{equation*}
\alpha\le 
\begin{cases}
1&\text{if $\Cl(Y)_\tors$ is of order $2$},
\\
2/9&\text{if $\Cl(Y)_\tors$ is of order $3$}.
\end{cases}
\end{equation*}
\end{corollary}

\section{Main construction}
\label{sect:constr}
\subsection{}
Let $X$ be a $\QQ$-Fano threefold.
For simplicity, we assume that the group $\Cl(X)$ is torsion free (this is the only case that we need in this paper). Denote $ q=\qQ(X)=\qW (X)$. Thus $-K_X\sim qA$ and 
$A$ is the ample generator of the group $\Cl(X)\simeq \ZZ$.

Consider a non-empty linear system $\MMM$ on $X$ without fixed components.
Let $c=\operatorname {ct} (X, \MMM)$ be the canonical threshold of the pair $(X, \MMM)$.
Consider a log crepant blowup $f: \tilde X \to X$ with respect to $K_X+c \MMM$.
One can choose $f$ so that
$\tilde X$ has only terminal $\QQ$-factorial singularities, i.e. $f$ is a divisorial extraction in the Mori category (see \cite{Corti1995a}, \cite{Alexeev-1994ge}).
Let $E$ be the exceptional divisor.
Write
\begin{equation} \label{equation-1}
\begin{array}{lll}
K_{\tilde X} &\qq & f^*K_X+\alpha E,
\\[2pt]
\tilde\MMM &\qq & f^*\MMM- \beta E.
\end{array}
\end{equation}
where $\alpha,\, \beta \in \QQ_{\ge0}$, and $\tilde\MMM$ is the birational transform of $\MMM$.
Then $c=\alpha/\beta$.
\begin{slemma}[see {\cite[Lemma 4.2]{Prokhorov2008a}}]
\label{lemma-cthreshold}
Let $P \in X$ be a point of index $r>1$. In a neighborhood of $P$ we can write
$\mathscr M \sim -tK_X$, where $0<t<r$. Then $c \le 1/t$ and so $\beta\ge t\alpha$.
\end{slemma}

Assume that the log divisor $-(K_X+c \MMM)$ is ample.
Run the log minimal model program with respect to $K_{\tilde X}+c \tilde\MMM$.
We obtain the following diagram (Sarkisov link, see \cite{Alexeev-1994ge}, \cite{Prokhorov2008a}, \cite{Prokhorov-e-QFano7})
\begin{equation}
\label{eq:sl}
\vcenter{
\xymatrix{
&\tilde X\ar[dl]_{f}\ar@{-->}[rr]^{\chi} && \bar X\ar[dr]^{\bar f}
\\
X &&&&\hat {X}
} }
\end{equation}
Here $\chi$ is a composition of $K_{\tilde X}+c \tilde\MMM$-log flips,
the variety $\bar X$ has only terminal $\QQ$-factorial singularities, 
$\uprho (\bar X)=2$, $\uprho (\hat X)=1$, and
$\bar{f}: \bar{X} \to \hat{X}$ is an extremal $K_{\bar{X}}$-negative Mori contraction.
In what follows, for the divisor (or linear system) $D$ on $X$
by $\tilde D$ and $\bar D$ we denote
proper transforms of $D$
on $\tilde X$ and $\bar{X}$ respectively.

If $|kA|\neq \emptyset$, we put $\MMM_k:=| kA|$
(is it possible that $\MMM_k$ has fixed components in general).
If $\dim \MMM_k=0$, then by $M_k$ we denote
a unique effective divisor $M_k \in \MMM_k$.
As in~\eqref {equation-1}, we write
\begin{equation} \label{equation-12}
\bar\MMM_k\qq f^*\MMM_k- \beta_k E.
\end{equation}

\subsection{}
Assume that the contraction $\bar{f}$ is birational.
Then $\hat{X}$ is a $\QQ $-Fano threefold.
In this case, we denote by $\bar F$ the
$\bar{f}$-exceptional divisor, by
$\tilde F \subset \tilde X$ its proper transform, $F:=f(\tilde F)$, and
$\hat{q}:=\qQ (\hat{X})$. Again we denote by $\hat D$ the proper (birational) transform of an object $D$ (resp. $\tilde D$, $\bar D$) on $X$ (resp. $\tilde X$, $\bar X$). Let $\Theta$ be an ample Weil divisor on $\hat{X}$
generating $\Cl(\hat{X})/\Cl(\hat{X})_\tors$.
Write
\begin{equation*}
\hat{E} \qq e\Theta,\qquad \hat\MMM_k \qq s_k\Theta,
\end{equation*}
where $e\in \ZZ_{>0}$, $s_k \in \ZZ_{\ge0}$. If $\dim \MMM_k=0$ and $\bar M_k=\bar F$
(i.e. a unique element $M_k$ of the linear system $\bar\MMM_k$ is the $\bar{f}$-exceptional divisor), we put
$s_k=0$.

\begin{slemma}
\label{lemma:genus}
If in the above notation $\alpha<1$, then $\g(\hat X)\ge\g(X)$. 
\end{slemma}
\begin{proof}
We have $a(E,|{-}K_X|)<1$. On the other hand, $0=K_X+|{-}K_X|$ is 
Cartier. Hence, $a(E,|{-}K_X|)\le 0$ and
$K_{\tilde X}+f^{-1}_*|{-}K_X|$ is linearly equivalent to a 
non-positive multiple of $E$. Therefore,
$f^{-1}_*|{-}K_X|\subset |{-}K_{\tilde X}|$ and so 
\begin{equation*}
\dim |-K_{\hat X}\rvert\ge 
\dim\lvert-K_{\bar X}\rvert=
\dim\lvert-K_{\tilde X}\rvert\ge\dim\lvert-K_X\rvert.\qedhere
\end{equation*}
\end{proof}

Note that in general, the group $\Cl(\hat X)$ can have torsions:
\begin{slemma}[see {\cite[Lemma 4.12]{Prokhorov2008a}}]
\label{lemma-torsion-d}
Write $F \sim dA$.
Then 
\[
\Cl(\hat{X})_\tors\simeq \ZZ/n\ZZ, \quad \text{where $n=d/e$.}
\]
\end{slemma}

\subsection{}
Assume that the contraction $\bar{f}$ is not birational.
In this case, $\Cl(\hat{X})$ has no torsion. Therefore, $\Cl(\hat{X})\simeq\ZZ$.
Denote by $\Theta$ the ample generator of $\Cl(\hat{X})$ and by
$\bar F$ a general geometric fiber.
Then $\bar F$ is either a smooth rational curve or a del Pezzo surface.
The image of the restriction map $\Cl(\bar{X}) \to \Pic (\bar F)$ is isomorphic to $\ZZ$.
Let $\Lambda$ be its ample generator.
As above, we can write
\begin{equation*}
-K_{\bar{X}} |_{\bar F} =-K_{\bar F} \sim \hat{q}\Lambda,\qquad \bar E |_{\bar F} \sim e \Lambda,\qquad \bar\MMM_k |_{\bar F} \sim s_k \Lambda,
\end{equation*}
where $\hat q,\, e\in \ZZ_{>0}$, $s_k \in \ZZ_{\ge0}$.

If $\hat{X}$ is a curve, then $\hat{q}\le 3$ and $\hat{X}\simeq\PP^1$.
If $\hat{X}$ is a surface, then $\hat{q}\le 2$.
In this case, $\hat{X}$ can have only Du Val singularities of type
$\mathrm {A_n}$ \cite[Theorem 1.2.7]{Mori-Prokhorov-2008}.

\begin{slemma}\label{lemma:not-birational:q}
If the contraction $\bar{f}$ is not birational and $\hat q>1$, then $X$ is rational.
\end{slemma}

\begin{proof}
Indeed, if $\hat X$ is a curve and $\hat q\ge 2$, then 
a general fiber $\bar F$ is a del Pezzo surface with divisible canonical class.
Then 
$\bar F$ is either a projective plane or a quadric. 
Clearly, $\bar X$ is rational in this case. Similarly, if $\hat X$ is a surface and $\hat q=2$,
then there is a divisor which is a generically section of $\bar f$ and $\bar X$ is again rational.
\end{proof}

\subsection{}
Since the group $\Cl(\bar X)$ has no torsion, the numerical equivalence of Weil divisors on $\bar X$ coincides with linear one. 
Hence the relations~\eqref {equation-1} and~\eqref {equation-12} give us
\begin{equation*}
k K_{\tilde{X}}+q \tilde\MMM_k \sim f^*( k K_X+q \MMM_k) +(k \alpha -q \beta_k) E\sim (k \alpha -q \beta_k) E
\end{equation*}
where $k \alpha -q \beta_k\in \ZZ$.
From this we obtain the following important equality which will be used throughout this paper:
\begin{equation} \label{equation-main}
k \hat{q}=q s_k+(q \beta_k-k \alpha) e.
\end{equation}

\subsection{}
Suppose that the morphism $\bar{f}$ is birational.
Similar to~\eqref {equation-1} and~\eqref {equation-12} we can write
\begin{equation*}
K_{\bar{X}}\qq \bar{f}^*K_{\hat{X}}+b \bar F, \quad
\bar\MMM_k\qq \bar{f}^*\hat\MMM_k-\gamma_k \bar F, \quad
\bar E\qq \bar{f}^*\hat{E}-\delta \bar F.
\end{equation*}
This gives us
\begin{equation*}
\begin{array}{lll}
s_kK_{\bar{X}}+\hat{q}\bar\MMM_k &\qq & (b s_k- \hat{q}\gamma_k) \bar F,
\\[2pt]
eK_{\bar{X}}+\hat{q}\bar E &\qq & (be-\hat{q}\delta) \bar F.
\end{array}
\end{equation*}
Taking proper transforms of these relations to $X$, we obtain 
\begin{eqnarray}
\label{eq:q=q:rh1}
-qs_k +\hat qk&=&ne(b s_k-\hat q\gamma _k),
\\[5pt]\label{eq:q=q:rh2}
-q&=&n(be -\hat q \delta).
\end{eqnarray}

\begin{scorollary}\label{corollary:gcd:qn}
If, in the above notation, $\gcd(n, q)=1$, then $\bar f(\bar F)$ is a point on $\hat X$ whose index is divisible by $n$.
\end{scorollary}

\begin{proof}
Indeed, either the discrepancy $b$ of $\bar F$ or the multiplicity $\delta$ is fractional and its denominator is divisible by $n$ according to~\eqref{eq:q=q:rh2}. 
\end{proof}

\section{$\QQ$-Fano threefolds of Fano index $7$ and large genus}
\label{sect:7}
Now we apply the techniques outlined in the previous section to $\QQ$-Fano threefolds of indices $\ge 7$.
The following result will be used in subsequent sections.
\begin{proposition}\label {proposition:ind=7}
Let $X$ be a $\QQ$-Fano threefold with $\qQ(X)=7$ and $\g(X)\ge 11$.
Then $X$ is rational.
\end{proposition}

\begin{proof}
By Proposition~\ref {proposition:torsions} the group $\Cl(X)$ is torsion free.
Assume that $X$ is not rational. 
According to \cite[Theorem~1.2, Proposition~2.1]{Prokhorov-e-QFano7} we have 
\begin{equation}\label{q=7:basket}
\B(X)=(2, 2, 3, r),
\end{equation}
where for $r$ there are only two possibilities:
\begin{gather}
\label{eq:q=7:1:case1}
r=5,\quad A^3=1/15,\quad \g(X)=11;
\\
\label{eq:q=7:1:case2}
r=12,\quad A^3=1/12,\quad \g(X)=13.
\end{gather}
In particular, $X$ has only cyclic quotient singularities. By the orbifold Riemann-Roch in both cases we have 
\begin{equation*}
\dim |kA|=k-1 \quad \text{for $k=1$, $2$, $3$.} 
\end{equation*}
Hence the linear system $|A|$ contains a unique irreducible surface $M_1$ and $|kA|$ has no fixed components for $k=2$ and $3$.

\subsection{}
Apply the construction~\eqref{eq:sl} with $\MMM=|3A|$. 
In a neighborhood of the point of index $r$ ($r=5$ or $12$) we have $\MMM\sim -tK_X$, where 
\begin{equation}
\label{eq:q=7:1:n}
t=\begin{cases}
4&\text{if $r=5$, } 
\\
9 & \text{if $r=12$}.
\end{cases} 
\end{equation}
Then by Lemma~\ref{lemma-cthreshold}
\begin{equation}
\label{eq:q=7:1:beta}
\beta_3\ge t\alpha.
\end{equation}
The relation~\eqref{equation-main} for $k=3$ has the form
\begin{equation}\label{eq:q=7:1:hatq}
3\hat q= 7 s_3+(7 \beta_3-3\alpha)e\ge 7 s_3+ (7t-3) \alpha e,
\end{equation}
where $\hat q\le 13$ by Proposition~\ref{prop:rat}.
If the contraction $\bar f$ is not birational, then $\hat q=1$ by Lemma~\ref{lemma:not-birational:q}. Hence, $\alpha\le 3/(7t-3)$.
On the other hand, 
\[
\alpha\ge 1/r> 3/(7t-3). 
\]
The contradiction shows that the contraction $\bar f$ must be birational. In particular, the movable linear system $\MMM$ is not contracted, i.e. 
\[
s_3\ge 1.
\]
\subsection{}
If $\alpha\ge 1$, then the inequality~\eqref{eq:q=7:1:hatq} and Proposition~\ref{prop:rat} give us successively
\begin{equation*}
3\hat q\ge 7 s_3+25 e,\qquad \hat q\ge 11,\qquad s_3\ge 5,\qquad \hat q>19, 
\end{equation*}
a contradiction. Taking~\eqref{q=7:basket} into account we see that 
$P:=f(E)$ is a non-Gorenstein point of $X$ and $f$ is the weighted blowup as in Proposition~\ref{prop:discrepancies}\ref{theorem-Kawamata-blowup}
(so-called \textit{Kawamata blowup}). In particular, $\alpha=1/\ind(X,P)$. In this case by Lemma~\ref{lemma:genus} we have
\begin{equation*}
\g(\hat X)\ge \g(X) \ge 11.
\end{equation*}
Since $\hat X$ is not rational, according to 
Proposition~\ref{prop:rat} we have 
\[
\hat q\le 7. 
\]
Note that $(7t-3) \alpha e\ge 5$. Then~\eqref{eq:q=7:1:hatq} implies 
\[
s_3\le 2.
\]

\subsection{Case: $\ind(X,P)=2$}
Then $\alpha=1/2$ and $\beta_3=1/2+m_3$, where $m_3\ge 2$ by~\eqref{eq:q=7:1:beta}.
We can rewrite~\eqref{eq:q=7:1:hatq} in the following form
\begin{equation*}
3\hat q= 7 s_3+(7 \beta_3-3\alpha)e= 2e + 7 (s_3+m_3e).
\end{equation*}
Since $\hat q\le 7$, this equation has no solutions.

\subsection{Case: $\ind(X,P)=3$}
Then, as above, $\alpha=1/3$, $\beta_3$ is an integer $\ge 2$, and~\eqref{eq:q=7:1:hatq} has the form
\begin{equation*}
3\hat q= 7 s_3+(7 \beta_3-3\alpha)e= -e + 7 (s_3+\beta_3e).
\end{equation*}
Again, there are no solutions.

\subsection{Case: $\ind(X,P)=r$, $r=5$ or $12$}
Then $\beta_1=t'/r+m_1$, where $m_1\ge 0$, and $t'=3$ if $r=5$ and $t'=7$ if $r=12$. 
The relation~\eqref{equation-main} for $k=1$ has the form
\begin{equation*}
7\ge \hat q= 7 s_1+(7 \beta_1-\alpha)e=4e + 7 (s_1+m_1e).
\end{equation*}
From this we obtain $s_1=0$ and $\hat q=4$. Then from~\eqref{eq:q=7:1:hatq} we obtain $s_3=1$.
Since $s_1=0$, the group $\Cl(\hat X)$ is torsion free by Lemma~\ref{lemma-torsion-d}.
Thus $\hat \MMM\sim 0$ and so $\dim |\Theta|\ge 2$. This contradicts Proposition~\ref{prop:rat}.
\end{proof}

\begin{corollary}
\label {corollary:ind=7}
Let $X$ be a $\QQ$-Fano threefold with $\qQ(X)=7$ and 
let $A$ be a Weil divisor such that $-K_X\qq 7A$ \textup(here we do not claim that $-K_X\sim 7A$\textup).
Assume that $\dim |2A|\ge 1$.
Then $X$ is rational. 
\end{corollary}

\begin{proof}
By Corollary ~\ref{cor:7tor} the group $\Cl(X)$ is torsion free.
Then a computer search gives us $\g(X)\ge 11$.
\end{proof}

\section{$\QQ$-Fano threefolds of Fano index $13$}
\label{sect:13}
\begin{proposition} \label{prop:13}
Let $X$ be a $\QQ$-Fano threefold with $\qQ(X)=13$. Then $X$ is rational.
\end{proposition}
\begin{proof}
By Proposition~\ref {proposition:torsions} the group $\Cl(X)$ is torsion free.
Assume that $X$ is not rational.
According to \cite{Prokhorov2008a} we have to consider only one case:
\begin{equation}\label{eq:q13:AB}
A^3=\textstyle \frac1 {210},\qquad \B=(2, 3, 3, 5, 7). 
\end{equation}
One can expect that \emph{all} the varieties of this type are hypersurfaces $X_{12}\subset \PP(3, 4, 5, 6, 7)$ (cf. \cite{Brown-Suzuki-2007j}), but this is not known.

By the orbifold Riemann-Roch,~\eqref{eq:q13:AB} implies that $|A|=|2A|=\emptyset$, 
the linear system $|kA|$ for $k=3$, $4$, $5$ contains a unique irreducible surface $M_k$ and for $k=6$, $7$, $8$ the linear system $|kA|$ is a pencil $\MMM_k$ without fixed components \cite[Proposition~3.6]{Prokhorov2008a}. 

\subsection{}
Apply the construction~\eqref{eq:sl} with $\MMM=|8A|$. Then near the point of index 7 we have $\MMM\sim -6K_X$. By Lemma~\ref{lemma-cthreshold}
\begin{equation}\label{eq:q=13:beta8}
\beta_8\ge 6\alpha.
\end{equation} 
The relation~\eqref{equation-main} for $k=8$ has the form
\begin{equation}\label{equation-13-q8}
8\hat q= 13 s_8+(13 \beta_8-8\alpha)e\ge 13 s_8+70e\alpha,
\end{equation}
where $\hat q\le 13$ by Proposition~\ref{prop:rat}.
Since $\alpha\ge 1/7$, we see that $\hat q>1$. By Lemma~\ref{lemma:not-birational:q} this implies that the contraction $\bar f$ is birational and so $s_8>0$. 
We also have 
\begin{equation*}
\tilde \MMM \qq \tilde M_3+\tilde M_5 + (\beta_3+\beta_5-\beta_8)E\qq 2\tilde M_4 +(2\beta_4-\beta_8)E,
\end{equation*}
where $\beta_3+\beta_5\ge \beta_8$ and $2\beta_4\ge \beta_8$.
Pushing forward this relation to $\hat X$ we obtain
\begin{equation*}
s_8=s_3+s_5 + (\beta_3+\beta_5-\beta_8)e=2s_4 +(2\beta_4-\beta_8)e.
\end{equation*}
Since the $\bar f$-exceptional divisor is irreducible, only one of the numbers $s_3$, $s_4$, $s_5$ can be equal to $0$. 
Therefore, 
\[
s_8\ge 2.
\]

\subsection{}
If $\alpha\ge 2/3$, then the relation~\eqref{equation-13-q8} gives us $\hat q\ge 10$. 
Then $\Cl(\hat X)$ is torsion free by Proposition~\ref{proposition:torsions}
and $\dim |k\Theta|\le 0$ for $k=1$, $2$, $3$ by Proposition~\ref{prop:rat}. Hence, $s_8\ge 4$. Then $\hat q\ge 13$ and so $s_8\ge 6$, $\hat q> 13$, a contradiction.
Therefore, $P:=f(E)$ is a non-Gorenstein point of $X$ and $f$ is the Kawamata blowup of $P$ by Proposition~\ref{prop:discrepancies}\ref{theorem-Kawamata-blowup}. In particular, $\alpha=1/\ind(X,P)$, where $\ind(X,P)=2$, $3$, $5$ or $7$.

\subsection{Case: $\ind(X,P)=2$} 
Then $\beta_8$ is an integer $\ge 3$ by~\eqref{eq:q=13:beta8}. The relation~\eqref{equation-13-q8} has the form
\begin{equation*}
8\hat q=-4e + 13(s_8+\beta_8e).
\end{equation*}
It has no solutions satisfying the inequalities $s_8\ge 2$, $\beta_8\ge 3$, $\hat q\le 13$. 

\subsection{Case: $\ind(X,P)=3$} 
Assume that $\ind(X,P)=3$. Then as above $\beta_8=2/3+m_8$, $m_8\ge 2$, and
\begin{equation*}
8\hat q=6e + 13(s_8+m_8e).
\end{equation*}
Again the equation has no suitable solutions.

\subsection{Case: $\ind(X,P)=5$} 
Then near the point of index $5$ we have $-K_X\sim \MMM_8$. Hence $\beta_8=1/5 +m_8$, where $m_8\ge 1$. The relation~\eqref{equation-13-q8} has the form
\begin{equation*}
8\hat q=e +13(s_8+m_8e).
\end{equation*}
We get only one solution: 
$\hat q=5$, $e=1$, $s_8=2$. Since $e=1$, we have $d=n$ by Lemma~\ref{lemma-torsion-d}.
Since $|A|=|2A|=\emptyset$, we have $d\ge 3$ and so $n=d=3$ by Proposition~\ref{proposition:torsions}. Thus $\Cl(\hat{X})_\tors\simeq\ZZ/3\ZZ$.
Then the image $\bar f(\bar F)$ is a non-Gorenstein point according to Corollary~\ref{corollary:gcd:qn}.
For $k=8$ the relation~\eqref{eq:q=q:rh1} yields $b\ge 7/3$. This contradicts Corollary~\ref{corollary:tor:discr}.

\subsection{Case: $\ind(X,P)=7$}
Finally we assume that $\ind(X,P)=7$.
Then $\beta_8=6/7 +m_8$, where $m_8\ge 0$. Hence,
\begin{equation}\label{eq:q=13:last}
8\hat q=10e +13(s_8+m_8e).
\end{equation}
If $e\le 2$, then the torsion part of $\Cl(\hat X)$ is non-trivial Lemma~\ref{lemma-torsion-d} because $|2A|=\emptyset$. By Proposition~\ref{proposition:torsions} we have $\hat q\le 7$ and then
\eqref{eq:q=13:last} has no solutions. Thus $e\ge 3$ and then there is only one possibility:
$\hat q=7$, $s_8=2$. Then $\hat X$ is rational by Corollary~\ref{corollary:ind=7}.
This concludes the proof of Proposition~\ref{prop:13}.
\end{proof}

\section{$\QQ$-Fano threefolds of Fano index $11$}
\label{sect:11}
\begin{proposition} \label{prop:11}
Let $X$ be a $\QQ$-Fano threefold with $\qQ(X)=9$. Then $X$ is rational.
\end{proposition}
\begin{proof}
By Proposition~\ref {proposition:torsions} the group $\Cl(X)$ is torsion free.
According to Proposition~\ref {prop:rat} and \cite{Prokhorov2008a} we have to consider only two cases:\par\smallskip\noindent
\begin{center}
\resizebox{\textwidth}{!}
{
\rm
\begin{tabularx}{\textwidth}{|l|X|c|c|c|c|c|c|c|c|}
\hline
&&&\multicolumn{6}{c|}{$\dim |kA|$}&
\\
\hhline{|~|~|~|------|~}
&$\B$&$A^3$& $|A|$&$|2A|$&$|3A|$&$|4A|$&$|5A|$&$|6A|$&$\g(X)$
\\[5pt]
\hline
\nr\label{case:q=11:non-empt} & $(2,5,7)$ & $1/70$ &
0
&0
&0
&1
&2
&3
&9
\\
\nr\label{case:q=11:empt} &
$(2,2,3,4,7)$ & $1/84$ &
$-1$&
$0$&
$0$&
$1$&
$1$&
$2$&
$7$
\\
\hline
\end{tabularx}
}
\end{center}
\par\smallskip\noindent
There are examples of varieties of these types: they are hypersurfaces $X_{12}\subset \PP(1, 4, 5, 6, 7)$ and $X_{10}\subset\PP(2, 3, 4, 5, 7)$ in cases~\ref{case:q=11:non-empt} and~\ref{case:q=11:empt}, respectively
\cite{Brown-Suzuki-2007j}.

\begin{scase}
\label{scase-MMM=5}
From the table above one can see that in both cases the linear systems $|kA|$ have no fixed components for $k=4$, $5$, $6$.
Apply the construction~\eqref{eq:sl} with $\MMM=|5A|$. Then near the point of index 7 we have $A\sim -2K_X$, $\MMM \sim -3 K_X$. By Lemma~\ref{lemma-cthreshold}
\begin{equation*}
\label{eq:q11:beta5}
\beta_5\ge 3\alpha.
\end{equation*}
The relation~\eqref{equation-main} for $k=5$ has the form
\begin{equation*}
5\hat q= 11 s_5+(11 \beta_5-5\alpha)e =-5\alpha e+11 (s_5+ \beta_5e) \ge 11 s_5+28\alpha e.
\end{equation*}
Assume that $X$ is not rational. Then $\hat q\le 11$ by Propositions~\ref{prop:rat} and~\ref{prop:13}.
\end{scase}

\subsection{}
Assume that $\alpha\ge 1$. 
Then $\hat q\ge 6$ and $\alpha$ is an integer by Proposition~\ref{prop:discrepancies}. Moreover, $\alpha=e=1$ and 
$s_5+ \beta_5e$ is also an integer. Hence, $\hat q\equiv -1\mod 11$. This contradicts~\eqref{equation-index}.
Therefore, $\alpha<1$ by Proposition~\ref{prop:discrepancies}\ref{theorem-Kawamata-blowup}. In particular, 
\begin{equation}\label{eq:q=11:alpha}
\alpha=1/r,\qquad r:=\ind(X,P)=2,\ 3,\ 4,\ 5\ \text{or}\ 7.
\end{equation} 

\subsection{}
Assume that $\bar f$ is not birational.
Since $\bar X$ is not rational by our assumptions, $\hat q=1$ (see Lemma~\ref{lemma:not-birational:q}). 
Then $s_5=0$ and 
$5= (11 \beta_5-5\alpha)e$, where $11 \beta_5-5\alpha\in \ZZ$. Then $\beta_5=l/r$, $l\in \ZZ$ and $l\ge 3$ by~\eqref{eq:q11:beta5}. Thus we can write $5r= (11l-5)e$. But this equation has no solutions satisfying \eqref{eq:q=11:alpha}. Therefore, the contraction $\bar f$ is birational. In particular, $s_5>0$.

\subsection{Cases~\ref{case:q=11:non-empt} and~\ref{case:q=11:empt} with $\ind(X,P)=2$} 
Then $\beta_5=1/2+m_5$, $m_5\ge 1$.
Thus~\eqref{equation-main} for $k=5$ has the form
\begin{equation*}
5\hat q =3 e+11 (s_5+m_5 e).
\end{equation*}
We get one possibility: $\hat q=5$, $e=1$, $s_5=1$. 

In the case~\ref{case:q=11:non-empt} the linear system $|A|$ contains a unique member $M_1$. Then~\eqref{equation-main} for $k=1$ has a similar form 
\begin{equation*}
5=\hat q=5e+ 11( s_1+m_1e),\quad m_1\ge0. 
\end{equation*}
We obtain $s_1=0$. So, $\Cl(\hat X)$ is torsion free by Lemma~\ref{lemma-torsion-d}. Since $\dim |\Theta|=2$, the variety $\hat X$ is rational
by Proposition~\ref{prop:rat}.

In the case~\ref{case:q=11:empt} the map $\bar f \comp \chi \comp f^{-1}$ contracts a divisor $F\sim dA$ with $d>1$ (because $|A|=\emptyset$). Since $e=1$, by Lemma~\ref{lemma-torsion-d} we have 
$\Cl(\hat X)_\tors\simeq \ZZ/n\ZZ$ with $n=d>1$. 
Apply~\eqref{eq:q=q:rh1}-\eqref{eq:q=q:rh2}.
Recall that $n\le 3$ (see Proposition~\ref{proposition:torsions}).
In particular, $\gcd (n, 11)=1$. Then the image $\bar f(\bar F)$ is a non-Gorenstein point according to Corollary~\ref{corollary:gcd:qn}. 
For $k=5$ the relation~\eqref{eq:q=q:rh1} yields $14=n(b -5\gamma _5)$.
According to Corollary~\ref{corollary:tor:discr} this is impossible.

\subsection{Cases~\ref{case:q=11:non-empt} and~\ref{case:q=11:empt} with $\ind(X,P)=7$.}
\label{case:q11:ind7}
Then $\beta_5=3/7+m_5$, $\beta_6=5/7+m_6$, where
$m_5, m_6\ge 0$. The relation~\eqref{equation-main} for $k=5$ and~$6$ has the form
\begin{equation}\label{eq:q=11:n7:5-6}
\begin{array}{rcl}
5\hat q&=&4 e+ 11( s_5+m_5e), 
\\
6\hat q&=&7e+ 11( s_6+m_6e). 
\end{array}
\end{equation} 
Here $s_5\le 3$ because $\hat q\le 11$.
By Proposition~\ref{prop:rat} we have $\hat q\neq 9$ because $\g(\hat X)\ge \g(X)\ge 7$.
Then the system of equations~\eqref{eq:q=11:n7:5-6} one has $\hat q=3e$, $s_5=s_6=e=1$ or $2$. 

Assume that $\hat q=6$ (and $e=s_5=s_6=2$). 
In the case~\ref{case:q=11:non-empt} we have
\begin{equation*}
6\bar M_1 +(6\beta_1-\beta_6)\bar E\qq\bar\MMM_6\qq 2\Theta,\quad 6\beta_1\ge \beta_6.
\end{equation*}
Hence the divisor $\bar M_1$ is contracted (otherwise the class of $\Theta$ in the group $\Cl(\hat{X})/\Cl(\hat{X})_\tors$ would be divisible). 
Since $e=2$, this contradicts Lemma~\ref{lemma-torsion-d}.
In the case~\ref{case:q=11:empt} from the relation
\begin{equation*}
3\bar M_2 +(3\beta_2-\beta_6)\bar E\qq\bar\MMM_6\qq 2\Theta.
\end{equation*}
we see that the divisor $\bar M_2$ must be contracted. Since $e=2$, 
the group $\Cl(\hat X)$ is torsion free by Lemma~\ref{lemma-torsion-d}.
Since $s_6=2$ and $\dim \MMM_6=2$, we have $\dim |2\Theta|\ge 2$. 
This contradicts Proposition~\ref{prop:rat}.

Finally, assume that $\hat q=3$ (and $e=s_5=s_6=1$). 
In the case~\ref{case:q=11:non-empt} we have
\begin{equation*}
6\bar M_1 +(6\beta_1-\beta_6)\bar E\qq\bar\MMM_6\qq \Theta,\quad 6\beta_1\ge \beta_6.
\end{equation*}
As above, the divisor $\bar M_1$ must be contracted and the group $\Cl(\hat X)$ is torsion free.
Since $s_6=1$ and $\dim \MMM_6=3$, we have $\dim |\Theta|\ge 3$. This contradicts Proposition~\ref{prop:rat}.

In the case~\ref{case:q=11:empt} we have
\begin{equation*}
3\bar M_2 +(3\beta_2-\beta_6)\bar E\qq 2\bar M_3 +(2\beta_3-\beta_6)\bar E\qq\bar\MMM_6\qq \Theta,
\end{equation*}
where $3\beta_2\ge \beta_6$, $2\beta_3\ge\beta_6$.
Since both $\bar M_2$ and $\bar M_3$ cannot be contracted simultaneously, this gives a contradiction.

\subsection{Case~\ref{case:q=11:empt} with $\ind(X,P)=3$}
Then $\beta_5=1/3+m_5$, $m_5\ge 1$.
Thus 
\begin{equation*}
5\hat q =2 e+11 (s_5+m_5 e) 
\end{equation*}
and we obtain $\hat q=7$ and $s_5\le 2$. 
Then $\hat X$ is rational by Corollary~\ref{corollary:ind=7}.

\subsection{Case~\ref{case:q=11:empt} with $\ind(X,P)=4$}
Then $\beta_5=3/4+m_5$, $m_5\ge 0$. If $m_5=0$, then $\mathrm{ct}(X,\MMM)=1/3$.
In this case $(X,\frac13 \MMM)$ is canonical and points of indices $4$ and $7$ are canonical centers. Then we can apply our construction~\eqref{eq:sl} starting with the point of index $7$, as in~\ref{case:q11:ind7}.
This gives a rationality construction.

Thus we assume that $m_5\ge 1$. The relation~\eqref{equation-main} for has the form
\begin{equation*}
5 \hat q=11(s_5 + m_5 e)+7e
\end{equation*}
and then $\hat q=8$, $s_5\le 2$.
By Proposition~\ref{prop:rat} the variety $\hat X$ is rational.

\subsection{Case~\ref{case:q=11:non-empt} with $\ind(X,P)=5$}
Then $\MMM$ is a Cartier at $P$ and so $\beta_5$ must be a positive integer. 
The relations~\eqref{equation-main} has the form 
\begin{eqnarray}
5 \hat q=11(s_5 + \beta_5 e)-e.
\end{eqnarray}
Since $\hat q\le 11$, this equation has no solutions. This concludes the proof of Proposition~\ref{prop:11}.
\end{proof}

\section{$\QQ$-Fano threefolds of Fano index $9$}
\label{sect:9}
\begin{proposition} \label{prop:9}
Let $X$ be a $\QQ$-Fano threefold with $\qQ(X)=9$. Then $X$ is rational.
\end{proposition}
\begin{proof}
By Proposition~\ref {proposition:torsions} the group $\Cl(X)$ is torsion free.
Assume that $X$ is not rational. 
According to \cite[Proposition~3.6]{Prokhorov2008a} we have to consider only one case: 
\begin{equation}\label{eq:q9:AB}
\B=(2,2,2,5,7),\qquad A^3=1/70. 
\end{equation}
By the orbifold Riemann-Roch~\eqref{eq:q9:AB} implies that 
\begin{equation*}
|A|=\emptyset, \quad
\dim |2A|=\dim |3A|=0,\quad
\dim |4A|=\dim |5A|=1.
\end{equation*}
Thus the linear system $|kA|$ contains a unique irreducible surface $M_k$ for $k=2$ and $3$ and $|kA|$ for $k=4$ and $5$ is a pencil without fixed components.

\subsection{}
Apply the construction~\eqref{eq:sl} with $\MMM=|5A|$. Then near the point of index ${7}$ we have $\MMM\sim -6K_X$. By Lemma~\ref{lemma-cthreshold}
\begin{equation*}
\beta_5\ge 6\alpha. 
\end{equation*}
The relation~\eqref{equation-main} for $k=5$ has the form
\begin{equation*}
5\hat q= 9 s_5+(9 \beta_5-5\alpha)e\ge 7 s_5+49 \alpha e.
\end{equation*}

\subsection{}
By Propositions~\ref{prop:rat},~\ref{prop:13}, and~\ref{prop:11} we have $\hat q\le 9$. Then, obviously,
$\alpha<1$. Therefore, $P:=f(E)$ is a non-Gorenstein point of $X$ by Proposition~\ref{prop:discrepancies}\ref{theorem-Kawamata-blowup} and $\alpha=1/\ind(X,P)$, where $\ind(X,P)=2$, $5$ or $7$.

\subsection{}
If $\bar f$ is not birational, then $\hat q=1$ by Lemma~\ref{lemma:not-birational:q} and so $s_5=0$, i.e. $\bar \MMM$ is $\bar f$-vertical.
Note that $9 \beta_5-5\alpha$ is an integer (because $9\MMM_5+5K_X$ is Cartier). Hence, $9 \beta_5-5\alpha=1$ or $5$.
Let $r:=\ind(X,P)$. Then $\beta_5=l/r$ for some $l$ and $9 l=r+5$ or $5(r+1)$. For $r=2$, $5$, $7$ this equation has no solutions.
The contradiction shows that $\bar f$ is birational. In particular, $s_5>0$.

\subsection{Case: $\ind(X,P)=2$}
Then $\beta_5=1/2+m_5$, $m_5\ge 3$ and the relation~\eqref{equation-main} for $k=5$ has the form
\begin{equation*}
5\hat q= 2e+9 (s_5+ m_5e).
\end{equation*}
Since $\hat q\le 9$, this is impossible.

\subsection{Case: $\ind(X,P)=5$}
Then $\beta_5$ is an integer $\ge 2$ and, as above,
\begin{equation*}
5\hat q= -e+9 (s_5+ \beta_5e).
\end{equation*}
We get one possibility: $\hat q=7$, $e=1$, $s_5+ \beta_5=4$. 
Since $|A|=\emptyset$, the group $\Cl(\hat X)_\tors$ is non-trivial by by Lemma~\ref{lemma-torsion-d}.
By Proposition~\ref{proposition:torsions} we have $\Cl(\hat X)_\tors\simeq \ZZ/2\ZZ$.
By Corollary~\ref{corollary:gcd:qn} 
the image $\bar f(\bar F)$ is a point of even index. The relation~\eqref{eq:q=q:rh1} for $k=5$ has the form
\[
35 -9s_5 =2(b s_5-7\gamma _5),\qquad b\ge (35 -9s_5)/2s_5\ge 17/4.
\]
Then we obtain a contradiction by Corollary~\ref{corollary:tor:discr}.

\subsection{Case: $\ind(X,P)=7$}
Then $\beta_5=6/7+m_5$, $m_5\ge 0$,
\begin{equation*}
5\hat q= 9 s_5+(9 \beta_5-5\alpha)e=7e +9 (s_5+ m_5e).
\end{equation*}
We get the following possibilities:
\begin{equation*}
(\hat q, e)= (5,1)\quad \text{or}\quad (6,3). 
\end{equation*}

If $\hat q=6$, then the group $\Cl(\hat X)$ is torsion free by Proposition~\ref{proposition:torsions}. 
Since $s_5+ 3m_5=1$, we have $s_5=1$. Hence, $\dim |\Theta|\ge 1$.
This contradicts Proposition~\ref{prop:rat}.

Consider the case $\hat q=5$. Then $s_5\le 2$ Since $e=1$ and $|A|=\emptyset$, by Lemma~\ref{lemma-torsion-d} we have 
$\Cl(\hat X)_\tors\simeq \ZZ/n\ZZ$ with $n=d>1$. 
Apply~\eqref{eq:q=q:rh1} with $k=5$. We obtain $25-9s_5\le ns_5b$ and so $b\ge 7/2n$. 
Since $n\le 3$, we get a contradiction by Corollary~\ref{corollary:tor:discr}.
This concludes the proof of Proposition~\ref{prop:9}.
\end{proof}

\section{$\QQ$-Fano threefolds of Fano index $8$}
\label{sect:8}
\begin{proposition} \label{prop:8}
Let $X$ be a $\QQ$-Fano threefold with $\qQ(X)=8$. Then $X$ is rational.
\end{proposition}
\begin{proof}
By Proposition~\ref {proposition:torsions} the group $\Cl(X)$ is torsion free.
Assume that $X$ is not rational. 
Using a computer search and taking Proposition~\ref{prop:rat} into account we obtain the following possibilities:
\par\smallskip\noindent
\begin{center}
{
\rm
{
\begin{tabularx}{\textwidth}{|X|X|c|c|c|c|c|X|}
\hline
&&\multicolumn{5}{c|}{$\dim |kA|$}&
\\
\hhline{|~|~|-----|~}
$\B$&$A^3$& $|A|$&$|2A|$&$|3A|$&$|4A|$&$|5A|$&$\g(X)$
\\[5pt]
\hline
$(7, 13)$ &
$4/91
$ &$0
$ &$0
$ &$1
$ &$2
$ &$3
$ &$11$
\\
$(5, 7)$ &
$1/35
$ &$0
$ &$0
$ &$1
$ &$2
$ &$3
$ &$8$
\\
$(3, 5, 11)$ &
$4/165
$ &$-1
$ &$0
$ &$0
$ &$1
$ &$2
$ &$6$\\\hline
\end{tabularx}
}
} 
\end{center}
\par\smallskip\noindent
Note that existence of varieties with $\B(X)=(7, 13)$ and $(5,7)$ is not known.
Varieties with $\B(X)=(3,5,11)$ can be realized as hypersurfaces $X_{12}\subset \PP(1,3,4,5,7)$ which are rational.
But again we do not know if this is the only family with corresponding invariants. 

Apply the construction~\eqref{eq:sl} with $\MMM=|4A|$.
Since $X$ is not rational by our assumption, we have $\hat q\le 8$ 
(see Propositions~\ref{prop:rat},~\ref{prop:13},~\ref{prop:11}, and~\ref{prop:9}).

\subsection{Case $\B(X)=(5,7)$}
In a neighborhood of the point of index $7$ we have $\MMM\sim -4K_X$. 
Thus by Lemma~\ref{lemma-cthreshold}
\begin{equation*}
\beta_4\ge 4\alpha.
\end{equation*}
The relation~\eqref{equation-main} for $k=4$ has the form
\begin{equation}\label{eq:q=8:first}
8\ge \hat q=2s_4+ (2\beta_4 - \alpha)e\ge 2s_4+ 7 \alpha e.
\end{equation}

We claim that the contraction $\bar f$ is birational. Indeed, otherwise $\hat q=1$ by Lemma~\ref{lemma:not-birational:q} and so $s_4=0$, i.e. $\bar \MMM$ is the pull-back of some linear system on $\hat X$. Since $\dim \bar \MMM=2$, $\dim \hat X\neq 1$ (otherwise $\bar \MMM=\bar f^*|2p|$, where $p$ is a point on $\hat X\simeq \PP^1$, and then $\bar M_2=\bar f^*p$ must be movable). 
Further, $4\bar M_1\sim \bar \MMM$ and so $\bar M_1$ is also the pull-back of some divisor, say $\Lambda$, on the surface $\hat X$. 
Thus $\bar M_1=\bar f^* \Lambda$ and $\bar \MMM=\bar f^* |4\Lambda|$. Clearly, $\Lambda$ is a generator of the group $\Cl(\hat X)$. Recall that $\hat{X}$ is a del Pezzo surface with at worst Du Val singularities of type
$\mathrm {A_n}$ \cite[Theorem 1.2.7]{Mori-Prokhorov-2008}. According to the classification (see e.g. \cite[Lemmas 3 \& 7]{Miyanishi-Zhang-1988}) for $\hat X$ there are only four possibilities:
\begin{equation*}
\PP^2,\qquad \PP (1,1,2),\qquad \PP (1,2,3)\quad \text{or}\quad \DP,
\end{equation*}
where $\DP$ is a del Pezzo surface of degree $5$ whose singular locus consists of one point of type $\mathrm {A_4}$.
Since $\dim |\bar M_1|=\dim |\bar M_2|=0$, the divisors $\Lambda$ and $2\Lambda$ are not movable. But one can easily check that $\dim |2\Lambda|>0$ in all cases.
The contradiction shows that the contraction $\bar f$ is birational. In particular, 
\[
s_4\ge 1.
\]
Then from \eqref{eq:q=8:first} we immediately see that $\alpha<1$. 
Therefore, $P:=f(E)$ is a non-Gorenstein point of $X$ and $\alpha=1/\ind(X,P)$, where $\ind(X,P)=5$ or $7$ (see Proposition~\ref{prop:discrepancies}\ref{theorem-Kawamata-blowup}).

\begin{scase}{\bf Subcase $\ind(X,P)=7$.}
Then we can write $\beta_1=1/7+m_1$ and $\beta_4=4/7+m_4$, where $m_1$ and $m_4$ are non-negative integers. We can rewrite the relation~\eqref{equation-main} for $k=1$ and $4$ as follows
\begin{equation*}
8\ge \hat q=8(s_1 + m_1 e)+e=2(s_4 + m_4 e)+e.
\end{equation*}
This yields $\hat q=e$ and $s_4=0$, a contradiction.
\end{scase}

\begin{scase}{\bf Subcase $\ind(X,P)=5$.}
As above, $\beta_4=3/5+m_4$, $\beta_1=2/5+m_1$, where $m_1\ge 0$ and $m_4\ge 1$. Therefore,
\begin{equation*}
8\ge \hat q=2(s_4 + m_4 e)+e=8(s_1 + m_1 e)+3e\ge 3.
\end{equation*}
This yields $s_1=0$ and $\hat q=3e$. Since $s_1=0$, by Lemma~\ref{lemma-torsion-d} we have $e=1$, $\hat q=3$, and $s_4=0$, a contradiction.
\end{scase}

\subsection{Case $\B(X)=(7, 13)$}
Near the point of index $13$ we have $\MMM\sim -7K_X$. 
Thus by Lemma~\ref{lemma-cthreshold}
\begin{equation}
\label{eq:q=8:beta:bc}
\beta_4\ge 7\alpha,\qquad \beta_1\ge\textstyle \frac 74\alpha.
\end{equation}
The relation~\eqref{equation-main} for $k=4$ has the form
\begin{equation}\label{eq:q=8:hatq:bc}
8\ge \hat q= 2 s_4+(2 \beta_4-\alpha)e\ge 2 s_4+13e\alpha.
\end{equation}
From this, one immediately sees that $\alpha<1$. 
Therefore, $P:=f(E)$ is a non-Gorenstein point of $X$ and $f$ is the Kawamata blowup of $P$ by Proposition~\ref{prop:discrepancies}\ref{theorem-Kawamata-blowup}. In particular, $\alpha=1/\ind(X,P)$, where $\ind(X,P)=7$ or $13$.

\begin{scase}{\bf Subcase $\ind(X,P)=13$.}
Then we can write $\beta_1=5/13+m_1$, where $m_1$ is a non-negative integer. Therefore,
\begin{equation*}
8\ge \hat q=8(s_1 + m_1 e)+3e\ge 3.
\end{equation*}
This gives us $s_1=0$ and $\hat q=3e$. Since $s_1=0$, Lemma~\ref{lemma-torsion-d} we have $e=1$, $\hat q=3$, and $\Cl(\hat X)$ is torsion free. 
Similarly, we can compute 
\begin{equation*}
\beta_5=12/13+m_5,\qquad 
5 \hat q=8(s_5 + m_5 e)+7e,\qquad s_5=1. 
\end{equation*}
Therefore, $\dim |\Theta|\ge \dim \MMM_5=3$.
This contradicts Proposition~\ref{prop:rat}.
\end{scase}

\begin{scase}{\bf Subcase $\ind(X,P)=7$.}
Then we can write $\beta_1=1/7+m_1$, where $m_1$ is a positive integer. Then
\begin{equation*}
8\ge \hat q=8(s_1 + m_1 e)+e\ge 9,
\end{equation*}
which is a contradiction.
\end{scase}

\subsection{Case $\B(X)=(3, 5, 11)$} 
Near the point of index $11$ we have $\MMM\sim -6K_X$. 
Thus by Lemma~\ref{lemma-cthreshold}
\begin{equation}
\label{eq:q=8:beta:c-3511}
\beta_4\ge 6\alpha.
\end{equation}
The relation~\eqref{equation-main} for $k=4$ has the form
\begin{equation}\label{eq:q=8:hatq:c}
8\ge \hat q= 2 s_4+(2 \beta_4-\alpha)e\ge 2 s_4+11e\alpha.
\end{equation}
From this we immediately see that $\alpha<1$.
Therefore, $\alpha=1/\ind(X,P)$, where $\ind(X,P)=3$, $5$ or $11$ (see Proposition~\ref{prop:discrepancies}\ref{theorem-Kawamata-blowup}).

\begin{scase}{\bf Subcase $\ind(X,P)=3$.}
Then we can write $\beta_4=2/3+m_4$, where $m_4\ge 2$. Therefore,
\begin{equation*}
8\ge \hat q=2(s_4 + m_4 e)+e\ge 5.
\end{equation*}
In particular, $\bar f$ is birational and $s_4>0$.
We get only one solution:
$\hat q=7$, $e=s_4=1$.
By Corollary~\ref{corollary:ind=7} the variety $\hat X$ is rational.
\end{scase}

\begin{scase}{\bf Subcase $\ind(X,P)=5$.}
Then we can write $\beta_2=4/5+m_2$ and $\beta_4=3/5+m_4$, where $m_2\ge 0$ and $m_4>0$. Therefore,
\begin{equation*}
8\ge \hat q=2(s_4 + m_4 e)+e=4(s_2 + m_2 e)+3e\ge 3.
\end{equation*}
In particular, $\bar f$ is birational and $s_4>0$.
We obtain $\hat q= 7$ and $s_4\le 2$.
Then $\hat X$ is rational again by Corollary~\ref{corollary:ind=7}.
\end{scase}

\begin{scase}{\bf Subcase $\ind(X,P)=11$.}
Then we can write $\beta_4=6/11+m_4$, where $m_4$ is a non-negative integer. Therefore,
\begin{equation}
\label{eq:q8c:4hat}
8\ge \hat q=2(s_4 + m_4 e)+e.
\end{equation}
Similarly, the relation~\eqref{equation-main} for $k=3$ has the form
\begin{equation}
\label{eq:q8c:3hat}
3 \hat q=8(s_3 + m_3 e)+7e,\qquad m_3\ge 0.
\end{equation}
One can see that there are only two solutions:
\begin{equation*}
(\hat q, e)= (5,1)\quad \text{or}\quad \text (7,3).
\end{equation*}

If $\hat q=7$, then 
by~\eqref{eq:q8c:4hat} we have $s_4= 2$. This contradicts Corollary~\ref{corollary:ind=7}.
Hence, $\hat q=5$ and $e=1$. 
Since $|A|=\emptyset$ and $e=1$, we have $\Cl(\hat X)_\tors\simeq \ZZ/n\ZZ$ with $1<n\le 3$
by Proposition~\ref{proposition:torsions}.
If $n=3$, then $s_3=0$ by Lemma~\ref{lemma-torsion-d}. Then $\bar f(\bar F)$ is a non-Gorenstein point by Corollary~\ref{corollary:gcd:qn}. 
The relation~\eqref{eq:q=q:rh1} gives us 
\[
3s_4 b\ge 20-8s_4, 
\]
where $s_4\le 2$ by~\eqref{eq:q8c:4hat}. Hence, $b\ge 2/3$.
This contradicts Corollary~\ref{corollary:tor:discr}.

Assume that $n=2$. Then $s_2=0$ by Lemma~\ref{lemma-torsion-d}.
The relation~\eqref{eq:q=q:rh1} for $k=3$ has the form 
\[
15-8s_3=2(bs_3 -5\gamma _3), 
\]
where $s_3=1$ by~\eqref{eq:q8c:3hat}. We see that $\bar f(\bar F)$ is a non-Gorenstein point and 
$b\ge 7/2$. Again, this contradicts Corollary~\ref{corollary:tor:discr}.
Proposition~\ref{prop:8} is proved.\qedhere
\end{scase}
\end{proof}
Now Theorem~\ref{theorem:main} follows from Propositions~\ref{prop:13}, \ref{prop:11}, \ref{prop:9}, and~\ref{prop:8}.

\newcommand{\etalchar}[1]{$^{#1}$}
\def\cprime{$'$}

% \providecommand*{\BibDash}{}
% \bibliography{all,prokho}
% \bibliographystyle{alpha}
\end{document}